\documentclass{amsart}
\usepackage{graphicx} % Required for inserting images
\usepackage{amsmath}
\usepackage{amssymb}
\usepackage{amsbsy}
\usepackage{amsfonts}%,srcltx}
\usepackage{curves} 
\usepackage{bm}
\usepackage[numbers,sort&compress]{natbib}  
\usepackage{comment}
\newtheorem{theorem}{Theorem}[section]%
\newtheorem{corollary}[theorem]{Corollary}%
\newtheorem{proposition}[theorem]{Proposition}%
\newtheorem{construction}[theorem]{Construction}%
\newcommand{\Sz}{\mathrm{Sz}}
\newcommand{\PSL}{\mathrm{PSL}}
\newcommand{\PGL}{\mathrm{PGL}}

  \def\G{\Gamma}

  \def\G{\Gamma}

\def\nd{\mathrel{\bigm|\kern-.7em/}}

\def\PSL{\hbox{\rm PSL}}   \def\PGL{\hbox{\rm PGL}}
 \def\P\GL{\hbox{\rm P\GL}}  
  \def\Aut{\hbox{\rm Aut}}

\title{Testing chirality on hypertopes}
\author{Wei-Juan Zhang}
\address{Wei-Juan Zhang, School of Mathematics and Statistics, Xi'an Jiaotong University, Xi'an 710049, China}
\email{weijuanzhang@mail.xjtu.edu.cn}

\author{Dimitri Leemans}
\address{Dimitri Leemans, D\'epartement de Math\'ematique, Universit\'e Libre de Bruxelles, C.P.216, Boulevard du Triomphe, 1050 Bruxelles, Belgium}
\email{leemans.dimitri@ulb.be}

\date{\today}

\keywords{hypertope, chirality, $C^+$-group, coset geometry}
\subjclass[2000]{51E24, 20F05}

\begin{document}
\begin{abstract}
    In this paper we give group-theoretical conditions on the maximal parabolic subgroups of a coset geometry for it to be a chiral hypertope, bypassing the need to construct the incidence graph of the coset geometry to determine whether or not it is a chiral hypertope. This result permits to study much larger coset geometries with computers and gives hope on proving theoretical results about coset geometries that are chiral hypertopes.
\end{abstract}

\maketitle

%%%%%%%%%%%
\section{Introduction}
%%%%%%%%%%%
In 2016, Maria Elisa Fernandes, Dimitri Leemans and Asia Weiss introduced the concept of hypertope\footnote{A hypertope is a thin and residually connected incidence geometry (see Section~\ref{sec:preliminaries} for definitions).} as a generalization of abstract polytopes~\cite{FLW16}.
They studied in particular two families of hypertopes, namely the regular ones (where the automorphism group has a unique orbit on the set of chambers) and the chiral ones (where the automorphism group has two orbits on the set of chambers with adjacent chambers in distinct orbits).
They showed among other things that the automorphism group of a regular hypertope has to be a $C$-group and that the automorphism group of a chiral hypertope has to be a $C^+$-group.

In the regular case, given a C-group, one can easily construct a coset geometry, and then, results from Buekenhout, Hermand, Dehon and Leemans permit to translate the geometric properties of thinnness, residual connectedness and flag-transitivity in group-theoretical properties (see~\cite{Dehon94} for instance). Such tests are available in the computational algebra software {\sc Magma}~\cite{Magma} with the built-in functions {\rm IsThin}, {\rm IsResiduallyConnected} and {\rm IsFTGeometry}. 
The group-theoretical conditions on the coset geometries made it possible to obtain several classification results, for instance for the groups $\PSL(2,q)$ and $\PGL(2,q)$~\cite{CJL15}, the Suzuki groups $\Sz(q)$~\cite{CL15} and the symmetric groups~\cite{FL18}.

Given a $C^+$-group, Fernandes, Leemans and Weiss explain how to construct a coset geometry, but there, 
unfortunately, the only way so far
to check if a $C^+$-group $(G^+, R)$ gives a chiral hypertope is to convert the coset geometry in an incidence geometry, that is to eventually construct the incidence graph $\mathcal G$ of the coset geometry $\Gamma$, then check that the group $G^+$ has two orbits on the maximal cliques of $\mathcal{G}$ (that correspond to the chambers of $\Gamma$) and that there is no automorphism of $G^+$ that fuses the two orbits. The construction of the incidence graph is highly inefficient and quickly reaches limits when one tries to construct chiral hypertopes with a computer as the incidence graph of a coset geometry is a graph whose vertex set is the set of all elements of the geometry. Moreover, this gap in the theory has stalled the theoretical study of chiral hypertopes for years.

A partial solution to the problem above was given in~\cite{LT} where Leemans and Tranchida showed that if $(G^+,R)$ is a $C^+$-group and if the incidence system $\Gamma$ associated to $(G^+,R)$ is chiral, then $\Gamma$ is also residually connected. But still, testing that the group has two orbits on the set of chambers of $\G$ remains a serious issue.

In this paper, we give the missing piece of the puzzle: we give group-theoretical conditions of the coset incidence system associated to a $C^+$-group $(G^+,R)$ for it to be a chiral hypertope, bypassing the costly need to construct the incidence graph of the coset incidence system and opening the possibility to check if a $C^+$-group arising from a group $G^+$ gives a chiral hypertope in a very efficient way for groups much larger than before. Our main result is the following theorem.

\begin{theorem}\label{Chiral hypertopes}
Let $(G^+, R)$ be a $C^+$-group of rank $r \ge 3$ with $I = \{0, 1, \dots, r-1\}$, and let $\G(G^+; (G^+_i)_{i \in I}) = \G(G^+, R)$ be the coset geometry associated to $(G^+, R)$ by Construction~\ref{hyper}. 
Then $\G$ is a chiral hypertope if and only if
for any $k\in I$ we have the following (where $J := I\setminus \{k\})$: 
\begin{enumerate} 
\item[(i)] $G^+$ acts transitively on the chambers of the $(r-1)$-truncation $^J\G(G^+; (G^+_i)_{i \in J})$,
\item[(ii)] $\bigcap_{t \in T} G^+_t = {1_{G^+}}$ for each subset $T \subset I$ with $|T| = r-1$, 
\item[(iii)] $\left| \bigcap_{j \in I \setminus {k}} (G^+_k G^+_j) \right| = 2 |G^+_k|$ for some $k \in I$, 
\item[(iv)] there is no automorphism of $G^+$ that inverts all elements of $R$.
\end{enumerate} 
\end{theorem}

The conditions we give in the theorem above
make it also easier to check theoretically that a given $C^+$-group gives a chiral hypertope.

The paper is organized as follows.
In Section~\ref{sec:preliminaries}, we give all the necessary definitions and notations to understand this paper, as well as some results needed in the proof of our main result.
In Section~\ref{sec:chamber}, we obtain group-theoretical conditions on the maximal parabolic subgroups of a coset incidence system for it to have two orbits on its chambers.
In Section~\ref{sec:chiralhypertope}, we put everything together in order to prove Theorem~\ref{Chiral hypertopes}.
Finally, in Section~\ref{code}, we provide a {\sc Magma} function to check if a $C^+$-group gives a coset incidence system that is a chiral hypertope.

%%%%%%%%%%%%%%%
\section{Background}
\label{sec:preliminaries}
%%%%%%%%%%%%%%%
Most of the definitions in this section come from~\cite{LT}.
\subsection{Incidence geometries}
An {\it incidence system}  is a 4-tuple $\Gamma := (X, *, t, I)$ such that
\begin{itemize}
\item $X$ is a set whose elements are called the {\it elements}, or {\it faces}, of $\Gamma$;
\item $I$ is a set whose elements are called the {\it types} of $\Gamma$;
\item $t:X\rightarrow I$ is a {\it type function}, associating to each element $x\in X$ of $\Gamma$ a type $t(x)\in I$;
\item $*$ is a binary relation on $X$ called {\em incidence}, that is reflexive, symmetric and such that for every $x,y\in X$, if $x*y$ and $t(x) = t(y)$ then $x=y$.
\end{itemize}
The {\it rank} of $\Gamma$ is the cardinality of $I$.
A {\it flag} is a set of pairwise incident elements of $\Gamma$. The {\em rank} of a flag $F$ is its cardinality and the {\em corank} of $F$ is the cardinality of $I\setminus t(F)$.
The {\it type} of a flag $F$ is $\{t(x) : x \in F\}$ and a flag of type $I$ is called a {\em chamber}.

An incidence system $\Gamma$ is a {\it geometry} or {\it incidence geometry} if every flag of $\Gamma$ is contained in a chamber.
An element $x$ is {\em incident} to a flag $F$, and we write $x*F$ for that, provided $x$ is incident to all elements of $F$.
If $\Gamma = (X, *, t, I)$ is an {\it incidence geometry} and $F$ is a flag of $\Gamma$,
the {\em residue} of $F$ in $\Gamma$ is the incidence geometry $\Gamma_F := (X_F, *_F, t_F, I_F)$ where $X_F := \{ x \in X : x * F, x \not\in F\}$;  $I_F := I \setminus t(F)$;  $t_F$ and $*_F$ are the restrictions of $t$ and $*$ to $X_F$ and $I_F$.

The {\it incidence graph} of $\Gamma$ is the graph whose vertex set is $X$ and where two distinct vertices are joined provided the corresponding elements of $\Gamma$ are incident.

An incidence system $\Gamma$ is {\em connected} if its incidence graph is connected. It is
{\em residually connected} when each residue of rank at least two of $\Gamma$ (including $\Gamma$ itself) has a connected incidence graph. 

An incidence system $\Gamma$ is {\it thin} (respectively {\em firm}) when every residue of rank one of $\Gamma$ contains exactly (respectively at least) two elements. 

Let $\Gamma =(X,*, t,I)$ be an incidence system.
An {\em automorphism} of $\Gamma$ is a 
permutation $\alpha$ of $X$ inducing a permutation of $I$ such that
\begin{itemize}
\item for each $x$, $y\in X$, $x*y$ if and only if $\alpha(x)*\alpha(y)$;
\item for each $x$, $y\in X$, $t(x)=t(y)$ if and only if $t(\alpha(x))=t(\alpha(y))$.
\end{itemize}
An automorphism $\alpha$ of $\Gamma$ is called {\it type preserving} when for each $x\in X$, $t(\alpha(x))=t(x)$.
The set of type-preserving automorphisms of $\Gamma$ is a group denoted by $Aut_I(\Gamma)$.
The set of automorphisms of $\Gamma$ is a group denoted by  $Aut(\Gamma)$.
A group $G\leq Aut_I(\Gamma)$ acts {\em flag-transitively} on $\Gamma$ if, for each $J\subseteq I$, the group $G$ is transitive on the set of flags of type $J$ of $\Gamma$. In this case, we also say that $\Gamma$ is {\em flag-transitive}.
The following proposition shows how, starting from a group $G$, we can construct an incidence system whose type-preserving automorphism group contains $G$.

\begin{proposition}(Tits, 1956)~\cite{Tits}\label{tits}
Let $n$ be a positive integer
and $I:= \{0,\ldots ,n-1\}$.
Let $G$ be a group together with a family of subgroups ($G_i$)$_{i \in I}$, $X$ the set consisting of all cosets $G_ig$ with $g \in G$ and $i \in I$, and $t : X \rightarrow I$ defined by $t(G_ig) = i$.
Define an incidence relation $*$ on $X\times X$ by:
\begin{center}
$G_ig_1 * G_jg_2$ if and only if $G_ig_1 \cap G_jg_2 \neq \emptyset$.
\end{center}
Then the 4-tuple $\Gamma := (X, *, t, I)$ is an incidence system having a chamber.
Moreover, the group $G$ acts by right multiplication on $\Gamma$ as a group of type preserving automorphisms.
Finally, the group $G$ is transitive on the flags of rank less than 3.
\end{proposition}

The incidence system constructed by the proposition above is called a {\em coset incidence system} and will be denoted by $\Gamma(G; (G_i)_{i\in I})$. It might not be a geometry, but if it is a geometry we call it a {\em coset geometry}.

Given a family of subgroups $(G_i)_{i\in I}$ and $J\subseteq I$, we define $G_J := \cap_{j \in J}G_j$. The subgroups $G_J$ are called the {\em parabolic subgroups} of the coset geometry $\Gamma(G; (G_i)_{i\in I})$. Moreover $G_j = G_{\{j\}}$ for each $j\in I$ and the subgroups $G_j$ are called the {\em maximal parabolic subgroups} of $\Gamma$.

In the case of flag-transitive coset geometries, there is an easy group-theoretical way to test residual connectedness, which was originally proved by Francis Buekenhout and Michel Hermand.

 \begin{theorem}\cite[Corollary 1.8.13]{BC}
\label{proposition:RC}
Suppose \(I\) is finite and let \(\Gamma= \Gamma(G,(G_i)_{i\in I})\) be a geometry over \(I\) on which \(G\) acts flag transitively. Then \(\Gamma\) is residually connected if and only if \(G_J = \langle G_{J\cup\{i\}}~|~i\in I\setminus J\rangle\) for every \(J\subseteq I\) with \(|I\setminus J| \ge 2\).
\end{theorem}

Leemans proved the following result in his master's thesis to check flag-transitivity recursively on the rank of the geometry.

\begin{theorem}\label{FTlee2}\cite{Leemans94}
 Let $\Gamma(G,\{G_0,\ldots,G_{r-1}\})$ be a flag-transitive coset geometry of rank $r$, let $H$ be a subgroup of $G$ and let $\Gamma'(H,\{G_0 \cap H, \ldots,$ $G_{r-1} \cap H\})$ be a flag-transitive geometry. Then the incidence systems $\Gamma_{(ij)}(G,$ $\{G_i,G_j,H\})$ are flag-transitive geometries $\forall ~ 0\leq i,j \leq r-1$, $i \neq j$ if and only if  the incidence system $\Gamma''(G,\{G_0,\ldots,G_{r-1},H\})$ is a flag-transitive geometry.
\end{theorem}

Two chambers of a geometry $\Gamma$ are {\em adjacent} if they differ by exactly one element. They are called {\em $i$-adjacent} if they differ in their elements of type $i$.

A geometry $\Gamma$ is {\em chiral} if $Aut_I(\Gamma)$ has two orbits on the chambers such that any two adjacent chambers lie in distinct orbits.
 
Observe that if a geometry is chiral, it is necessarily thin as if a rank one residue contains more than two elements, this contradicts chirality.
 
A \emph{hypertope} is a thin, residually connected incidence geometry. 
A hypertope $\Gamma$ is {\em regular} if $\Gamma$ is a flag-transitive geometry.
A hypertope $\Gamma$ is {\em chiral} if $\Gamma$ is a chiral geometry.
 
 Let $\Gamma(X,*,t,I)$ be a thin geometry and $i\in I$. If $C$ is a chamber of $\Gamma$, we let $C_i$ denote the chamber {\em $i$}-adjacent to $C$, that  is, the chamber that differs from $C$ only in its $i$-face. 
 
\subsection{$C$-groups and regular hypertopes}

Given a regular hypertope $\Gamma$ and a chamber $C$ of $\Gamma$, for each $i\in I$ let $\rho_i$ denote the automorphism mapping $C$ to $C_i$.
Then $\{\rho_0,\ldots, \rho_{n-1}\}$ is a generating set for $Aut_I(\Gamma)$ and $G_i=\langle \rho_j\,|\, j\neq i\rangle$ is the stabilizer of the $i$-face of $C$. 
Moreover $(Aut_I(\Gamma),\{ \rho_0,\ldots, \rho_{n-1}\})$ is a \emph{$C$-group}  \cite[Theorem~4.1]{FLW16}, that is, $\{ \rho_0,\ldots, \rho_{n-1}\}$ is a set of involutions generating $Aut_I(\Gamma)$ and satisfying the following condition, called the \emph{intersection condition}.
$$\forall I, J \subseteq \{0, \ldots, n-1\},
\langle \rho_i \mid i \in I\rangle \cap \langle \rho_j \mid j \in J\rangle = \langle \rho_k \mid k \in I \cap J\rangle.$$

From a $C$-group we can get a hypertope when the incidence system arising from Proposition~\ref{tits} is flag-transitive, as shown in the following theorem.

\begin{theorem}\cite[Theorem 4.6]{FLW16}\label{theorem46}
Let $G=\langle \rho_0, \ldots, \rho_{n-1}\rangle$ be a $C$-group of rank $n$ and let $\Gamma := \Gamma(G;(G_i)_{i\in I})$ with $G_i := \langle \rho_j | j \in I\setminus \{i\} \rangle$ for all $i\in I:=\{0, \ldots, n-1\}$.
If $G$ is flag-transitive on $\Gamma$, then $\Gamma$ is a regular hypertope.
\end{theorem}

\subsection{C$^+$-groups}\label{section6}

We now consider another class of groups 
from which we will be able to construct hypertopes. These hypertopes may or may not be regular. In the latter case, they will be chiral. 

Consider a pair $(G^+,R)$ with $G^+$ being a group and $R:=\{\alpha_1, \ldots, \alpha_{r-1}\}$ a set of generators of $G^+$.
Define $\alpha_0:=1_{G^+}$
 and $\alpha_{ij} := \alpha_i^{-1}\alpha_j$ for all $0\leq i,j \leq r-1$.
Observe that $\alpha_{ji} = \alpha_{ij}^{-1}$.
Let $G^+_J := \langle \alpha_{ij} \mid i,j \in J\rangle$ for $J\subseteq \{0, \ldots, r-1\}$.

If the pair $(G^+,R)$ satisfies condition (\ref{IC+}) below called the {\em intersection condition} IC$^+$, we say that $(G^+,R)$ is a {\em $C^+$-group}.
\begin{equation}\label{IC+}
\forall J, K \subseteq \{0, \ldots, r-1\}, with \;|J|, |K| \geq 2,
G^+_J \cap G^+_K = G^+_{J\cap K}.
\end{equation}
It follows immediately from the intersection condition IC$^+$ that, if $(G^+,R)$ satisfies IC$^+$, then $R$ is an independent generating set for $G^+$ (meaning that for every $i=1, \ldots, r-1$, $\alpha_i \not\in \langle \alpha_j : j \neq i\rangle$).

We now explain how to construct a coset geometry from a group and an independent generating set of this group.

\begin{construction}\cite[Construction  8.1]{FLW16}\label{hyper}
Let $I=\{1,\ldots, r-1\}$,  $G^+$ be a group and $R:=\{\alpha_1,\ldots, \alpha_{r-1}\}$ be an independent generating set of $G^+$.
Define $G^+_i := \langle \alpha_j | j \neq i \rangle$ for $i=1, \ldots, r-1$ and $G^+_0 := \langle \alpha_1^{-1}\alpha_j | j \geq 2 \rangle$.
The incidence system $\Gamma(G^+,R) := \Gamma(G^+; (G^+_i)_{i\in \{0,\ldots,r-1\}})$ constructed using Tits' algorithm (see Proposition~\ref{tits}) is the {\em incidence system associated to the pair} $(G^+,R)$. 
\end{construction}

If the incidence system $\Gamma(G^+,R)$ is a chiral hypertope, then $(G^+,R)$ is necessarily a $C^+$-group by the following theorem.

Given a chiral hypertope $\Gamma(X,*,t,I)$ (with $I:=\{0, \ldots, r-1\}$) and its automorphism group $G^+:=Aut_I(\Gamma)$, pick a chamber $C$.
For any pair $i\neq j \in I$, there exists an automorphism $\alpha_{ij}\in G^+$ that maps $C$ to $(C^i)^j$.
Also, $C\alpha_{ii}=(C^i)^i=C$ and $\alpha_{ij}^{-1}=\alpha_{ji}$.
We define 
\[\alpha_i := \alpha_{0i}\;(i=1\ldots, r-1)\]
and call them the {\it distinguished generators} of $G^+$ with respect to $C$.

\begin{theorem}\cite[Theorem 7.1]{FLW16}\label{cplusgroup}
Let $I:=\{0, \ldots, r-1\}$ and let $\Gamma$ be a chiral hypertope of rank $r$. Let $C$ be a chamber of $\Gamma$.
The pair $(G^+,R)$ where $G^+=Aut_I(\Gamma)$ and $R$ is the set of distinguished generators of $G^+$ with respect to $C$ is a $C^+$-group.
\end{theorem}

\begin{corollary}\cite[Corollary 7.2]{FLW16}
The set $R$ of Theorem~\ref{cplusgroup} is an independent generating set for $G^+$. 
\end{corollary}

The following theorem shows that if $(G^+,R)$ is a C$^+$-group and the incidence system associated to $(G^+,R)$ is chiral, then the intersection condition $IC^+$ of the C$^+$-group implies the residual connectedness of $\Gamma$.

\begin{theorem}\cite[Theorem 1.1]{LT}\label{LTRC}
Let $G$ be a group and $R$ be a set of generators of $G$ such that \((G^+,R)\) is a \(C^+\)-group.
Let \(\Gamma= \Gamma(G^+,(G_i^+)_{i\in I})\) be the incidence system associated to \(G^+\). Then, if \(\Gamma\) is chiral, it is residually connected.
\end{theorem}

%%%%%%%%%%%%%%%
\section{Coset geometries with two orbits on their chambers}
\label{sec:chamber}
%%%%%%%%%%%%%%%

In~\cite[Theorem 8.1]{FLW16}, it was proven that any rank $n-2$ truncation of a chiral hypertope of rank $n\geq 3$ is a flag-transitive geometry. This result was improved in~\cite[Proposition 4.2]{LT}.

\begin{theorem}~\cite[Proposition 4.2]{LT}\label{FTtruncation} 
Let $\G(X, *, t, I)$ be a chiral hypertope of rank $|I| \ge 3$. Then, any truncation of $\G$ of rank at most $|I| -1$ is a flag-transitive geometry. \end{theorem}
\begin{comment}
\begin{proof}
Let $k\in I$ and $J:=I\backslash \{k\}$.
Let $^J\G$ be the truncation of type $J$ of $\Gamma$.
As $\G$ is a chiral hypertope, $^J\G$ has at most two orbits on its set of chambers (for otherwise, $\G$ would necessarily have more that two orbits on its flags, contradicting the fact that it is chiral).
Take any two chambers say $C_1$ and $C_2$ of $^J\G$. They are also flags of $\G$. As $\G$ is a thin geometry, each of these flags belong to exactly two chambers of $\G$. More precisely, there exist $w, x,y, z \in X$, all of type $k$, such that $C_1\cup \{w\}$, $C_1\cup \{x\}$, $C_2\cup \{y\}$, $C_2\cup \{z\}$ are chambers of $\G$.
As $\G$ is a chiral hypertope, the chambers $C_2\cup\{y\}$ and $C_2\cup\{z\}$ are adjacent and lie in two distinct orbits. Therefore, there exists $g\in \Aut_I(\G)$ that maps the chamber $C_1\cup \{w\}$ to $C_2\cup\{y\}$ or $C_2\cup\{z\}$. The same $g$ necessarily maps $C_1$ to $C_2$ in $\G$ and also in $^J\G$. 
Hence $^J\G$ is flag-transitive.
\end{proof}    
\end{comment}

The above result is very strong. It means that if we want to check whether a coset geometry $\G$ is a chiral hypertope, we need every truncation of corank one of $\G$ to be a flag-transitive geometry.
We now determine what extra conditions a coset geometry needs to satisfy in order to have two orbits on its set of chambers.

\begin{theorem}\label{ChambersOrbits} 
Let $r\geq 3$ be an integer. Let $I := \{0, \ldots, r-1\}$.
Let $\G(G; (G_i)_{i\in I})$ be a coset geometry. 
Let $k\in I$ and $J := I\backslash \{k\}$.
Suppose that $G$ acts transitively on the chambers of the $(r-1)$-truncation $^J\G(G; (G_i)_{i \in J})$.
Then
$G$ has two orbits on the chambers of $\G$ if and only if
$\bigcap_{j \in J} (G_{k} G_j) = G_{k} \left( \bigcap_{j \in J} G_j\right) \cup G_{k} w \left( \bigcap_{j \in J} G_j\right)$
 for some $w \in \bigcap_{j \in J} (G_{k} G_j)\setminus G_{k} \left( \bigcap_{j \in J} G_j\right)$.
 \end{theorem}

\begin{proof}
$\fbox{$\Rightarrow$}$
By hypothesis, $G$ has a unique orbit on flags of type $J$.
Let $F_k :=\{G_j : j \in J\}$ be the base flag of type $J$.
The set of cosets of $G_k$ that are incident to $F_k$ is the set whose union is $\bigcap_{j \in J} (G_{k} G_j)$.
Also $G$ has two orbits on the chambers. One orbit is the orbit of the base chamber $\{G_i : i \in I\}$. It is the orbit obtained by the cosets $G_{k} \left( \bigcap_{j \in J} G_j\right)$ of $G_k$.
Since $G$ has a second orbit, there must be an element $w \in \bigcap_{j \in J} (G_{k} G_j)\setminus G_{k} \left( \bigcap_{j \in J} G_j\right)$ such that the second orbit is the set of cosets $G_{k} w \left( \bigcap_{j \in J} G_j\right)$ of $G_k$ hence the claimed equality is satisfied.

%%%%%%
$\fbox{$\Leftarrow$}$
Let $B := \bigcap_{j \in J} G_j$. Since $G$ is flag-transitive on $^J\G$, the stabilizer of the flag $C := \{G_j : j \in J\}$ is $B$. 
The set of elements of type $k$ incident to $C$ is given by the cosets $G_k h$ such that $G_k h$ is incident to $G_j$ for all $j \in J$. This is equivalent to $h \in G_k G_j$ for all $j \in J$, so the set is $\{ G_k h : h \in \bigcap_{j \in J} (G_k G_j) \}$.

By assumption, $\bigcap_{j \in J} (G_k G_j) = G_k B \cup G_k w B$ with $w \notin G_k B$, and the union is disjoint. We now examine the action of $B$ on this set of type $k$ points.

%We consider the subset $G_k B$. For any $h \in G_k B$, we have $h = g b$ for some $g \in G_k$ and $b \in B$. Then $G_k h = G_k g b = G_k b$ since $g \in G_k$. Thus, $\{ G_k h : h \in G_k B \} =\{ G_k b : b \in B \}$. The group $B$ acts on this set by right multiplication, and this action is transitive because for any $G_k b_1, G_k b_2 \in \{ G_k b : b \in B \}$, the element $b_1^{-1} b_2 \in B$ maps $G_k b_1$ to $G_k b_1 (b_1^{-1} b_2) = G_k b_2$.
%
%Similarly, for any $h \in G_k w B$, we have $h = g w b$ for some $g \in G_k$ and $b \in B$. Then $G_k h = G_k g w b = G_k w b$ and $\{ G_k h : h \in G_k w B \} = \{ G_k w b : b \in B \}$. For any $G_k w b_1, G_k w b_2 \in \{ G_k w b : b \in B \}$, the element $b_1^{-1} b_2 \in B$ maps $G_k w b_1$ to $G_k w b_2$, which implies that the action of group $B$ on this set is also transitive.

Consider $h \in G_k B$. Then $h = g b$ for some $g \in G_k$ and $b \in B$, so $G_k h = G_k g b = G_k b$ since $g \in G_k$. Thus, the chambers containing $C$ from $G_k B$ are of the form $\{G_0, G_1, \dots, G_{k-1}, G_k b, G_{k+1}, \dots, G_{r-1}\}$ with $b \in B$. The group $B$ acts transitively on these chambers by right multiplication. 
Indeed, for any two chambers with $G_k b_1$ and $G_k b_2$, the element $b_1^{-1} b_2 \in B$ maps the first to the second, which implies that the group $B$ acts transitively on the set and these chambers are in the same orbit under the action of $B$.
We note that in particular, for any $b \in B$, the element $b^{-1} \in B$ maps the chamber with $G_k b$ to the chamber with $G_k$. 

On the other hand, for any $h \in G_k w B$, we have $h = g w b$ for some $g \in G_k$ and $b \in B$. Then $G_k h = G_k g w b = G_k w b$, and the chambers from $G_k w B$ are of the form $\{G_0, G_1, \dots, G_{k-1}, G_k w b, G_{k+1}, \dots, G_{r-1}\}$ with $b \in B$. For any $b_1, b_2 \in B$, the element $b_1^{-1} b_2$ maps the chamber with $G_k w b_1$ to the one with $G_k w b_2$. Thus, these chambers form another $B$-orbit with the chamber with $G_k w$ lies in.

The two orbits are disjoint. Indeed, suppose that there exist $b_1, b_2 \in B$ such that $G_k b_1 = G_k w b_2$. Then $w b_2 b_1^{-1} \in G_k$, which implies $w \in G_k B$, a contradiction. Therefore, $B$ has exactly two orbits on the set of type $k$ points incident to $C$.

By hypothesis, $G$ is flag-transitive on $^J\G$, hence it acts transitively on flags of type $J$ of $\G$. Thus, for any flag $C_1$ of type $J$ of $\G$, there exists $g \in G$ such that $g(C) = C_1$. The set of elements of type $k$ incident to $C_1$ is in bijection with the set of elements incident to $C$ via $g$. 
%The number of $G$-orbits on chambers of $\G$ is equal to the number of $B$-orbits on the type $k$ points incident to $C$, because each chamber is of the form $C_1 \cup {G_k h}$ and $G$ acts transitively on flags of type $J$. Hence, $G$ has exactly two orbits on chambers.
Therefore, the number of orbits of $G$ on the chambers of $\G$ is equal to the number of orbits of $G$ on the set of chambers  containing $C$. Furthermore, for any two elements of type $k$, say $G_k h_1$ and $G_k h_2$, incident with $C$, there exists $g\in G$ mapping $C\cup\{G_k h_1\}$ to $C\cup\{G_k h_2\}$ if and only if there is $g\in B$ such that $G_k h_1\cdot g=G_k h_2$. The number of orbits of $G$ on the chambers containing $C$ is equal to the number of orbits of $B$ on the elements of type $k$ that are incident with $C$.
Since $B$ has two orbits on this set, it follows that $G$ has exactly two orbits on the chambers of $\G$.
\end{proof}

\begin{corollary} 
Let $r\geq 3$ be an integer. Let $I := \{0, \ldots, r-1\}$.
Let $\G(G; (G_i)_{i\in I})$ be a coset geometry. 
Let $k\in I$ and $J := I\backslash \{k\}$.
Suppose that $G$ acts transitively on the chambers of the $(r-1)$-truncation $^J\G(G; (G_i)_{i \in J})$.
Suppose moreover that
$$\left| \bigcap_{j \in J} (G_{k} G_j) \right| = 2 \left| G_{k} \left( \bigcap_{j \in J} G_j \right) \right|$$
and that one of the following holds:
\begin{enumerate}
\item[(i)] $\bigcap_{j \in J} G_j \trianglelefteq \bigcap_{j \in J} (G_k G_j)$ or
\item[(ii)] $G_k \trianglelefteq \bigcap_{j \in J} (G_k G_j)$.
\end{enumerate}
Then $\G$ has exactly two orbits on chambers. 
\end{corollary}

\begin{proof}
By hypothesis, the group $G$ acts flag-transitively on the chambers of $^J\G$. We take a standard flag $C = \{G_j : j \in J\}$ of type $J$. 
Any chamber of $\G$ containing $C$ has the form 
$$\{G_0, G_1, \dots, G_{k-1}, G_kh, G_{k+1}, \dots, G_{r-1}\}$$
for some $h \in G$. As shown above, the incidence conditions require that $h$ must lie in the set $\bigcap_{j \in J}(G_kG_j)$.
Moreover, the set of elements of type $k$ of $\Gamma$ that are incident with $C$ is $\{G_{k} h : h \in  \bigcap_{j \in J}(G_kG_j)\}$. Elements $G_kh_1$ and $G_kh_2$ are the same if and only if $G_kh_1=G_kh_2$.

We notice that $B= \bigcap_{j \in J}G_j$ is the stabilizer of $C$ in $G$. For any $b \in B$, by right multiplication, $b$ preserves the flag $C$ and sends the chamber corresponding to $G_k h$ to that corresponding to $G_k h b$. Therefore, the action of $B$ on the set of chambers containing $C$ is equivalent to right multiplication on the cosets ${G_k h}$ by elements of $B$. The orbits under this action are precisely the double cosets $G_k \backslash  \bigcap_{j \in J}(G_kG_j) / B$. As a consequence, the group $\bigcap_{j \in J}(G_kG_j)$ can be divided into disjoint union of sets $G_k h B$ for some $h\in \bigcap_{j \in J}(G_kG_j)$.

Suppose first that (i) holds, that is $B \trianglelefteq \bigcap_{j \in J} (G_k G_j)$. Then for any $h \in \bigcap_{j \in J} (G_k G_j)$, we have $h B h^{-1} = B$, so $|G_k \cap h B h^{-1}| = |G_k \cap B|$. Therefore, each double coset $G_k h B$ has size $|G_k| |B| / |G_k \cap B| = |G_k B|$. Since the total size is $2 |G_k B|$, there must be exactly two double cosets, i.e., $\bigcap_{j \in J} (G_k G_j) = G_k B \cup G_k w B$ for some $w \notin G_k B$.

Suppose next that (ii) holds, that is $G_k \trianglelefteq \bigcap_{j \in J} (G_k G_j)$, for any $h \in \bigcap_{j \in J} (G_k G_j)$, we have $h G_k h^{-1} = G_k$, so $G_k \cap h B h^{-1} = h (G_k \cap B) h^{-1}$, and thus $|G_k \cap h B h^{-1}| = |G_k \cap B|$. Therefore, each double coset $G_k h B$ has size $|G_k| |B| / |G_k \cap B| = |G_k B|$. Again, there are exactly two double cosets.

In both cases, the conditions of Theorem \ref{ChambersOrbits} are satisfied, so $G$ has exactly two orbits on the chambers of $\G$. 
\end{proof}

A direct corollary follows.
\begin{corollary} \label{trivialstab}
Let $r\geq 3$ be an integer. Let $I := \{0, \ldots, r-1\}$.
Let $\G(G; (G_i)_{i\in I})$ be a coset geometry. 
Let $k\in I$ and $J := I\backslash \{k\}$.
Suppose that $G$ acts transitively on the chambers of the $(r-1)$-truncation $^J\G(G; (G_i)_{i \in J})$.
%Let $\G(G; (G_i)_{i\in\{0, 1, \dots, r-1\}})$ be a coset geometry of rank $r\ge 3$. If for some $k\in \{0, 1, \dots, r-1\}$, $G$ is flag-transitive on the $(r-1)$-truncation $\G_k$, and for $J = \{0, 1, \dots, r-1\} \setminus \{k\}$, 
If 
 $\bigcap_{j \in J} G_j = \{1_G\}$ and $\left| \bigcap_{j \in J} (G_k G_j) \right| = 2 |G_k|$ hold,
then $\G$ has exactly two orbits on the set of chambers of $\Gamma$. 
\end{corollary}

%%%%%%%%%%%%%%%
\section{Chiral hypertopes}
\label{sec:chiralhypertope}
%%%%%%%%%%%%%%%
We now put everything together to get group-theoretical conditions on the coset incidence system associated to a $C^+$-group for it to give a chiral hypertope, which is our main result -- Theorem \ref{Chiral hypertopes}.
\begin{proof}[Proof of Theorem~\ref{Chiral hypertopes}]
If $\G$ is a chiral hypertope, by Theorem~\ref{FTtruncation}, condition (i) is true and since $\G$ is thin, condition (ii) is satisfied (for otherwise there would be a rank one residue containing more than two elements).
Moreover, condition (iii) must be satisfied for, otherwise, $\G$ would again not be thin.
Condition (iv) follows from the fact that $\G$ is chiral, and hence, two adjacent chambers have to be in distinct orbits. If there is an automorphism of $G^+$ that inverts all elements of $R$, this automorphism will swap adjacent chambers.
%%%%

Suppose there exists a $k\in I$ such that conditions (i) to (iv) are satisfied and let us prove now that $\G$ is then a chiral hypertope.
Condition (ii) implies that for any flag of rank $r-1$, the stabilizer in $G^+$ is trivial. Combined with conditions (i) and (iii) and Corollary \ref{trivialstab}, it implies that $G^+$ has exactly two orbits on the chambers of $\G$ and that every residue of rank one contains exactly two elements. Therefore, $\G$ is thin.

Condition (iv) ensures that $\G$ is chiral, as it prevents the existence of an automorphism of $G^+$ that would make the geometry flag-transitive. The residual connectedness of $\G$ follows from Theorem~\ref{LTRC}, which states that chiral incidence systems associated to $C^+$-groups are necessarily residually connected. 

The fact that if conditions (i) to (iv) are satisfied for $k$ then they must also be satisfied for any other element of $I$ is obvious.
\end{proof}

In other words, if the geometry $\G$ associated with a $C^+$-group is a hypertope with two orbits, along with Condition (iv), by~\cite[Theorem 8.2]{FLW16}, the geometry $\G$ is a chiral hypertope. 

Observe that by Theorem~\ref{FTtruncation}, we can pick any $k\in I$ in the theorem above. 
%This means that, in order to speed up computations, a function that checks if a $C^+$-group gives a chiral hypertope should take that into account and pick the $k$ for which $G_k$ is the smallest.

Observe that if conditions (i), (ii) and (iii) are satisfied but condition (iv) is not, then $(G^+,R)$ is a $C^+$-group whose coset incidence system is actually a regular hypertope. The automorphism group of that hypertope is then a group having $G^+$ as subgroup and being twice as large as $G^+$.
\section{{\sc Magma} code}\label{code}
\begin{verbatim}
/*
The following function checks if a pair (G,R) where G is a group and R 
a set of generators of G is a C+-group
*/
function IntersectionProperty(G,R)
  // given a group G and a set R of generators, 
  // checks if (G,R) is a C+-group
  S := [[Id(G)] cat R];
  for i := 2 to #S[1] do
    T := [];
    for j := 1 to #S[1] do
      Append(~T,S[1][i]^-1*S[1][j]);
    end for;
    Append(~S,T);
  end for;
  for I in Subsets({1..#S}) do
    if #I ge 2 then
      for J in Subsets({1..#S}) do
        if #J ge 2 then
          if #(sub<G|{S[x,y] :  x,y in I| x ne y}> meet 
               sub<G|{S[x,y] : x,y in J|x ne y}>) ne 
             #sub<G|{S[x,z] : x,z in I meet J| x ne z}> then 
               return false;
          end if;
        end if;
      end for;
    end if;
  end for;
  return true;
end function;

/*
The following function construct a coset incidence system from a 
C+-group using Construction 2.5.
*/
function CPlusGroupToCosetGeometry(G,S);
 T:=Set(S);
 Gis:= {sub<G|x> : x in Subsets(T,#S-1)} join 
       {sub<G| {S[1]^-1*S[j] : j in {2..#S}}>};
 return CosetGeometry(G,Gis);
end function;

/*
The following function checks if the coset incidence system constructed 
from a C+-group (G,R) is a chiral hypertope. It returns true if it is 
and false with a number stating which property has failed in Theorem 4.1 
if it is not.
*/
function IsChiralHypertope(G,R)
  cg := CPlusGroupToCosetGeometry(G,R);
  G := Group(cg);
  m := MinimalParabolics(cg);
  orders := {#x : x in m};
  if #orders ne 1 or not(1 in orders) then
    // condition (ii) is not satisfied
    return false, 2;
  end if;
  M:=MaximalParabolics(cg);
  I := [2..#M];
  t := CosetGeometry(G,{M[i] : i in I});
  if not(IsFTGeometry(t)) then
    // condition (i) is not satisfied
    return false, 1;
  end if;
  cardGk := #M[1];
  flags := M[1]*M[2];
  for l := 3 to #M do
    flags := flags meet M[1]*M[l];
  end for;
  if #flags ne 2*cardGk then
    //condition (iii) is not satisfied
    return false, 3;
  end if;
  if IsHomomorphism(sub<G|R>, sub<G|R>, [R[i]^-1 : i in [1..#R]]) then
    // condition (iv) is not satisfied, we have a regular hypertope
    return false, 4;
  end if;
  return true;
end function;
\end{verbatim}
\section{Acknowledgements}
Dimitri Leemans acknowledge financial support from the Actions de Recherche Concertées of the Fédération Wallonie-Bruxelles, and Wei-Juan Zhang is supported by the National Natural Science Foundation of China (Grant No. 12201486) and the Fundamental Research Funds for the Central Universities (Grant No. xxj032025053).

%\newpage

\end{document}